\documentclass[conference]{IEEEtran}
\IEEEoverridecommandlockouts
\usepackage{cite}
\usepackage{amsmath,amssymb,amsfonts}
\usepackage{algorithmic}
\usepackage{graphicx}
\usepackage{textcomp}
\usepackage{xcolor}
\def\BibTeX{{\rm B\kern-.05em{\sc i\kern-.025em b}\kern-.08em
    T\kern-.1667em\lower.7ex\hbox{E}\kern-.125emX}}
    
\usepackage{amstext,amsthm,mathrsfs}
\usepackage{multirow}
\usepackage{url} 
\usepackage{bbm}
\usepackage{bm}
\usepackage{booktabs}
\usepackage{enumitem}
\usepackage{array}
\usepackage[utf8]{inputenc}

\usepackage[utf8]{inputenc} 
\usepackage[T1]{fontenc}
\usepackage{url}
\usepackage{ifthen}
\usepackage{cite}

\newcommand{\CADD}{{\mathsf{CADD}}}
\newcommand{\WADD}{{\mathsf{WADD}}}

\providecommand{\scrP}{\mathscr{P}}
\providecommand{\calP}{\mathcal{P}}
\providecommand{\supp}{\mathrm{supp}}
\providecommand{\wass}{\mathsf{W}}

\providecommand{\ind}{\mathbb{I}}
\DeclareMathOperator*{\esssup}{ess\,sup}

\DeclareMathOperator*{\argmin}{arg\,min}

\def\cT{{\cal T}}
\def\bE{{\mathbb{E}}}

\newtheorem{theorem}{Theorem}

\newtheorem{remark}{Remark}
\newtheorem{definition}{Definition}

\begin{document}
\title{Minimax Robust Quickest Change Detection using Wasserstein Ambiguity Sets}

\author{\IEEEauthorblockN{
Liyan Xie}
\IEEEauthorblockA{\textit{School of Data Science} \\
\textit{The Chinese University of Hong Kong, Shenzhen}\\
Shenzhen, China \\
xieliyan@cuhk.edu.cn}
}

\maketitle

\begin{abstract}
We study the robust quickest change detection under unknown pre- and post-change distributions. To deal with uncertainties in the data-generating distributions, we formulate two data-driven ambiguity sets based on the Wasserstein distance, without any parametric assumptions. The minimax robust test is constructed as the CUSUM test under least favorable distributions, a representative pair of distributions in the ambiguity sets. We show that the minimax robust test can be obtained in a tractable way and is asymptotically optimal. We investigate the effectiveness of the proposed robust test over existing methods, including the generalized likelihood ratio test and the robust test under KL divergence based ambiguity sets.
\end{abstract}



\section{Introduction}\label{sec:intro}

Quickest change detection aims to detect a potential change-point from sequential data and is widely applicable in signal processing and statistical problems \cite{tartakovsky2014sequential,Siegmund1985,veeravalli2013quickest}. 
Classical approaches, such as the well-known cumulative sum (CUSUM) test \cite{page-biometrica-1954}, are usually designed for cases where the pre- and post-change distributions are exactly known. When the post-change distribution is unknown, the generalized likelihood ratio (GLR) test \cite{lai-ieeetit-1998} is commonly used, in which the post-change distributions are sequentially estimated based on maximum likelihood.

However, the maximum likelihood estimate may deviate significantly from the true parameter if we only have limited data samples or the observations are contaminated \cite{huber1965robust,gao2018robust}. We aim to overcome this limitation by considering a robust quickest change detection problem by constructing {\it ambiguity sets} for the distribution estimates. The goal is to find the minimax robust test that minimizes the worst-case detection delay over the ambiguity sets \cite{fauss2021minimax}. In \cite{huber1965robust}, it is proved that an exact minimax robust optimal test does not hold for the robust sequential detection problem in general. Therefore, most work focus on finding the asymptotically optimal test \cite{sun2021data,molloy2017misspecified}.

The minimax robust change detection has been studied in \cite{molloy2017misspecified} and \cite{unnikrishnan2011minimax} with two ambiguity sets that are given in a priori, for the pre- and post-change distributions, respectively. In \cite{unnikrishnan2011minimax}, it is proved that under the joint stochastic boundedness condition on the pre- and post-change distributional ambiguity sets, the detection rule based on least favorable distributions (LFDs) are minimax robust under several performance metrics. Although the joint stochastic boundedness condition can be satisfied and verified for several classical types of ambiguity sets, it is difficult to verify for modern types of ambiguity sets, e.g., the KL ambiguity sets. Later in \cite{molloy2017misspecified}, the problem is solved by proving a weaker condition on the ambiguity sets, and asymptotic optimal solutions are proposed.
A recent work \cite{hare2021toward} studies the change detection with uncertain distributions from the Bayesian perspective by applying the uncertain likelihood ratio \cite{hare2020non} test. However, the posterior prediction distribution cannot be calculated when the parametric model is unknown or insufficient to model the data distribution. 


The main contribution of this work is a non-parametric method for minimax robust quickest change detection based on Wasserstein ambiguity sets \cite{villani2003topics}. The key advantage is that the proposed method does not require complete knowledge about pre- and post-change distributions and parametric assumptions. Moreover, the resulting LFDs from the Wasserstein ambiguity sets are proved to be efficiently solvable, and thus the proposed test can be applied to a wide range of applications.



The remainder of this paper is organized as follows. Section \ref{sec:formulation} details the problem set-up, including the performance criteria and the construction of the ambiguity sets. Section \ref{sec:theory} derives a tractable formulation to find the LFDs and the minimax optimal test. Section \ref{sec:numerical} demonstrates the proposed detection 
procedure using synthetic data. Section \ref{sec:conclusion} concludes the paper with possible future directions.

\section{Problem Setup} \label{sec:formulation}

The quickest change detection problem can be formulated as follows. Given observations $\{x_t,\ t=1,2,\ldots\}$ in the sample space $\mathcal X$, we aim to detect the change-point $\tau$ at which the data-generating distribution changes from $\mu$ to $\nu$:
\begin{equation}
\begin{array}{ll}
x_t  \stackrel{\text{iid}}{\sim} \mu, &t = 1,2,\ldots,\tau-1, \\
x_t \stackrel{\text{iid}}{\sim} \nu, &t = \tau,\tau+1,\ldots
\end{array} 
\label{eq:hypothesis}
\end{equation}
We consider the case where $\tau$ is {\it unknown} but is a {deterministic} value.
An important quantity for the detection problem \eqref{eq:hypothesis} is the Kullback-Leibler (KL) divergence defined as follows.
\begin{definition}[KL divergence \cite{kullback1951information}]
The KL divergence between two probability distributions $\nu$ and $\mu$ is:
\[
\mathsf{KL}(\nu||\mu) = \int  \{\log (d\nu(x)/d\mu(x))\}d\nu(x).
\]
\end{definition}

Let $\scrP(\mathcal X)$ denote the
family of all probability distributions supported on the sample space $\mathcal X$. Assume there exists a probability space $(\mathcal X,\mathcal F,\mathbb P_\tau^{\mu,\nu})$ where $\mathbb P_{\tau}^{\mu,\nu}$ denotes the probability measure when the change-point equals to $\tau$ and the pre- and post-change probability measures being $\mu$ and $\nu$, respectively. In particular, $\mathbb P_\infty^\mu$ and $\mathbb E_\infty^\mu$ denote the probability and expectation when there is no change-point (i.e., $\tau=\infty$) and the pre-change distribution being $\mu$. Similarly, $\mathbb P_0^\nu$ and $\mathbb E_0^\nu$ denote the probability and expectation when all samples are generated from the post-change distribution $\nu$. 

 Our goal is to detect the unknown change-point $\tau$ as quickly as possible while at the same time keeping the false alarm rate below a pre-specified level. Usually, the detection is performed by designing a {\em stopping time} on the data sequence \cite{xie2021sequential}. A stopping time with respect to the random data sequence $\{x_t\}_t$ is a random variable $T$ such that for any $n$, the event $\{T=n\}$ belongs to the sigma-algebra generated by $\{x_1, \ldots, x_n\}$.


\subsection{Performance Criteria}

We typically focus on two criteria to measure the performance of a stopping time $T$. One is the average run length (ARL) used to measure the average time between consecutive false alarms, defined as $\mathbb E_\infty^\mu[T]$.
Usually we impose certain lower bound $\gamma$ on the ARL and only consider the stopping times satisfying $\mathbb E_\infty^\mu[T]\geq \gamma$.
The other criteria is the detection delay.
There are two main measures for the detection delay, the Lorden's measure \cite{Lorden1971} and the Pollak's measure \cite{poll-astat-1985}.

The Lorden's measure for detection delay is defined as the worst-case average detection delay (WADD), which is the supremum of the average delay conditioned on the worst-case historical data and  change-point:
\begin{equation}\label{eq:WADDdef}
\WADD^{\mu,\nu}(T)\! = \! \underset{n \geq 1}{\operatorname{\sup}} \esssup \ \mathbb E_n^{\mu,\nu}\left[(T-n)^+| X_1, \dots, X_{n-1}\right].
\end{equation}
A less conservative characterization of detection delay is proposed by Pollak \cite{poll-astat-1985} as the conditional average detection delay (CADD) conditioned on the event that $\{T\geq n\}$:
\begin{equation}\label{eq:PollakCADDDef}
\CADD^{\mu,\nu}(T) = \underset{n \geq 1}{\operatorname{\sup}}\ \mathbb E_n^{\mu,\nu}[T-n| T\geq n].
\end{equation} 

\subsection{Uncertainty Model}
Consider the case when the pre- and post-change probability measure $\mu$ and $\nu$ in \eqref{eq:hypothesis} are {\it unknown}. This typically happens in real data applications, especially for data with complex structures or of high-dimensionality.
To deal with the uncertainties in distributions, we construct two ambiguity sets $\calP_{\mu_0},\calP_{\nu_0}$ for pre- and post-change distributions, respectively. 

Assume we have a nominal distribution $\mu_0$ and $\nu_0$ for pre- and post-change, and the ambiguity sets $\calP_{\mu_0},\calP_{\nu_0}$ are the collection of probabilities measures that are close to $\mu_0, \nu_0$ with respect to certain divergence measures $D(\cdot,\cdot)$:
\begin{equation}\label{eq:ambiguity}
\begin{aligned}
\calP_{\mu_0} &= \{\mu\in \scrP(\mathcal X): D(\mu,\mu_0) \leq r_1\},\quad \\ 
\calP_{\nu_0} &= \{\nu\in \scrP(\mathcal X): D(\nu,\nu_0) \leq r_2\},
\end{aligned}
\end{equation}
where $r_1,r_2\geq 0$ are the radius parameter controlling the size of ambiguity sets. 
Some commonly used divergence measures $D(\cdot,\cdot)$ include the KL divergence \cite{gul2017minimax,Levy2009}, Total-Variation distance \cite{huber1965robust,huber1973minimax,fauss2020minimax}, Wasserstein metric \cite{gao2016distributionally,esfahani2015data,gao2018robust}, etc. 

In this paper, we consider a fully {\it data-driven} and {\it non-parametric} setting where (i) the nominal distribution is set as the {\it empirical distribution} from historical data, and (ii) the ambiguity sets are constructed using the {\it Wasserstein distance}.

In the data-driven case, suppose we have a set of training samples $\{x_1,\ldots,x_{n_1}\}$ that are i.i.d. sampled from the pre-change regime, and $\{y_1,\ldots,y_{n_2}\}$ that are i.i.d. sampled from the post-change regime, the nominal distribution is set as the empirical distribution of those historical samples, i.e., $\mu_0 =(\sum_{i=1}^{n_1} \delta_{x_i})/n_1$, $\nu_0=(\sum_{i=1}^{n_2} \delta_{y_i})/n_2$, where $\delta_{x}$ denotes the Dirac point mass concentrated on $x$ for each $x\in\mathcal X$, i.e., $\delta_{x}(A) = \ind\{x \in A\}$ for any Borel measurable set and $\ind\{\cdot\}$ is the indicator function.
\begin{remark}
The historical data used here is additional available data before we start the detection procedure for problem \eqref{eq:hypothesis}. If we have no access to historical data in post-change regime beforehand, we may consider construing the post-change ambiguity sets adaptively with sequential observations, and the detailed discussion will be left for future work.
\end{remark}
Moreover, the Wasserstein metric we use in this paper is defined as follows. 
For two given distributions $P,Q\in \scrP(\mathcal X)$, their Wasserstein distance (of order 1) equals to \cite{villani2003topics}:
\begin{equation}\label{eq:was}
\wass(P,Q) :=\min_{\Gamma\in\Pi(P,Q)}\bE_{(\omega,\omega')\sim\Gamma} \left[c(\omega,\omega')\right], 
\end{equation}
where $c(\cdot,\cdot): \mathcal X \times \mathcal X \rightarrow \mathbb R_+$ is a metric, and $\Pi(P,Q)$ is the set of all joint probability distributions on $\mathcal X\times\mathcal X$ with marginal distributions $P$ and $Q$. 
%

%

Substitute the nominal distribution as the empirical distributions and the divergence measure as the Wasserstein metric, we construct the ambiguity sets as in \eqref{eq:ambiguity}:
\begin{equation}\label{eq:was_set_vanilla}
\begin{aligned}
\calP_{\mu_0} &= \{\mu\in \scrP(\mathcal X): \wass(\mu,\mu_0) \leq r_1\},\\
\calP_{\nu_0} &= \{\nu\in \scrP(\mathcal X): \wass(\nu,\nu_0) \leq r_2\}.
\end{aligned}
\end{equation}

\subsection{Minimax Robust Change Detection}

Under the ambiguity sets \eqref{eq:was_set_vanilla}, we aim to find the robust optimal stopping time that solves the following problem:
\begin{equation}\label{eq:Lorden}
\inf_{T\in C(\gamma,\calP_{\mu_0})} \sup_{\mu\in\calP_{\mu_0},\nu\in \calP_{\nu_0}}\WADD^{\mu,\nu}(T),
\end{equation}
where $C(\gamma,\calP_{\mu_0})$ is the set containing all stopping times $T$ that satisfies $\mathbb E_\infty^\mu[T] \geq \gamma, \forall \mu\in\calP_{\mu_0}$.
Similarly, the corresponding problem defined using $\CADD$ is:
\begin{equation}\label{eq:Pollak}
\inf_{T\in C(\gamma,\calP_{\mu_0})} \sup_{\mu\in\calP_{\mu_0},\nu\in \calP_{\nu_0}}\CADD^{\mu,\nu}(T).
\end{equation}
In general, it may be challenging to exactly solve the problems \eqref{eq:Lorden} and \eqref{eq:Pollak}. Therefore, asymptotically optimal solutions for the above problems are often investigated in practice. A solution $T_0\in C(\gamma,\calP_{\mu_0})$ is called first-order asymptotic optimal \cite{xie2021sequential} for \eqref{eq:Lorden} (and similarly defined for \eqref{eq:Pollak}) if: 
\[
\lim_{\gamma\rightarrow \infty} \frac{ \sup_{\mu\in\calP_{\mu_0},\nu\in \calP_{\nu_0}}\WADD^{\mu,\nu}(T_0)}{\inf_{T\in C(\gamma,\calP_{\mu_0})} \sup_{\mu\in\calP_{\mu_0},\nu\in \calP_{\nu_0}}\WADD^{\mu,\nu}(T)} = 1.
\]
\begin{remark}
The choice of the radius $r_1$, $r_2$ is crucial for the minimax detection problem. There is a tradeoff between model robustness and detection performance. A large radius will lead to a more robust detection but also a larger detection delay. Empirically, we can use cross-validation to set the radius. Theoretically, we may analyze the concentration of the Wasserstein distance to determine the appropriate radius \cite{fournier2015rate}. 
\end{remark}

\section{Optimal stopping time and theoretical guarantee}\label{sec:theory}

In this section, we derive the asymptotic optimal stopping time that solves the problem \eqref{eq:Lorden} and \eqref{eq:Pollak}. Based on previous results established in \cite{molloy2017misspecified}, the optimal stopping time can be constructed based on a pair of distributions in the ambiguity sets $(\calP_{\mu_0},\calP_{\nu_0})$, which are called the {\it least favorable distributions} (LFD). We first list the conditions to find such a pair of LFDs and show that they can be efficiently solved under the Wasserstein ambiguity sets \eqref{eq:was_set_vanilla}. Then we construct the optimal stopping time, which is a CUSUM test \cite{page-biometrica-1954} based on LFDs. 

\subsection{Least Favorable Distributions}
There are two types of conditions for finding the LFDs, which can be viewed as a representative pair of distributions within $(\calP_{\mu_0},\calP_{\nu_0})$ on which the stopping time reaches the worst-case performance.
%
The first condition, joint stochastic boundedness, was proposed in \cite{unnikrishnan2011minimax} as follows.
\begin{definition}[Joint stochastic boundedness \cite{unnikrishnan2011minimax}]
A pair of ambiguity sets $(\calP_{\mu_0},\calP_{\nu_0})$ is jointly stochastically bounded by the pair of distributions $(\tilde \mu,\tilde \nu)$ if $\forall \nu \in \calP_{\nu_0}$,
\[
\mathbb P^{\tilde \nu}\left(\log \frac{d\tilde \nu}{d\tilde \mu}(X) \geq x\right) \leq \mathbb P^{\nu}\left(\log \frac{d\tilde \nu}{d\tilde \mu}(X) \geq x\right), \forall x \in \mathbb R,
\]
and $\forall \mu \in \calP_{\mu_0}$,
\[
\mathbb P^{\mu}\left(\log \frac{d\tilde \nu}{d\tilde \mu}(X) \geq x\right) \leq \mathbb P^{\tilde \mu}\left(\log \frac{d\tilde \nu}{d\tilde \mu}(X) \geq x\right), \forall x \in \mathbb R.
\]
\end{definition}

This condition was later relaxed by \cite{molloy2017misspecified} as follows. 
\begin{definition}[Weak stochastic boundedness \cite{molloy2017misspecified}]
A pair of ambiguity sets $(\calP_{\mu_0},\calP_{\nu_0})$ is weakly  stochastically bounded by the pair of distributions $(\tilde \mu,\tilde \nu)$ if
\begin{equation}\label{eq:cond1}
  \mathsf{KL}(\tilde \nu ||\tilde \mu) \leq \mathsf{KL}(\nu ||\tilde \mu) - \mathsf{KL}(\nu || \tilde \nu), \forall \nu \in \calP_{\nu_0},  
\end{equation}
and
\begin{equation}\label{eq:cond2}
\mathbb E^\mu\left[\frac{d\tilde \nu}{d\tilde \mu}(X)\right]\leq \mathbb E^{\tilde \mu}\left[\frac{d\tilde \nu}{d\tilde \mu}(X)\right]=1, \forall \mu \in \calP_{\mu_0}.
\end{equation}
\end{definition}
In \cite{molloy2017misspecified}, it was shown that finding the pair of distributions that satisfies the weak stochastically boundedness condition is equivalent to finding the pair of distributions that minimizes the pairwise KL divergence between ambiguity sets. More specifically, the LFDs $(\tilde \mu,\tilde \nu)$ satisfying \eqref{eq:cond1} is a solution to:
\begin{equation}\label{eq:kl_min}
\min_{\mu \in \calP_{\mu_0},\nu \in \calP_{\nu_0}} \mathsf{KL}(\nu||\mu).
\end{equation}


Our main finding is that under the Wasserstein ambiguity sets \eqref{eq:was_set_vanilla}, the pair of distributions such that the sets are weakly stochastic bounded can be found through the following convex optimization problem efficiently. Denote $n=n_1+n_2$, $\{z_1,\ldots,z_n\}$ as the union of pre- and post-change historical data in the order of $\{x_1,\ldots,x_{n_1},y_1,\ldots,y_{n_2}\}$, and $\bm{1}_n$ as a $n$-dimensional column vector with all entries equal to one.
\begin{theorem}[LFD]\label{thm:lfd}
The pair of LFD solving \eqref{eq:kl_min} can be found by the following finite-dimensional convex program
 \begin{equation}\label{eq:wass vs wass:dual}
  \begin{aligned}
    &&\min_{\substack{p_1,p_2\in\mathbb R_+^{n}\\\Gamma_1,\Gamma_2\in\mathbb R_+^{n\times n}}}\!         &&& \!
     \sum_{l=1}^{n} p_2^l\log(p_2^l/p_1^l) \\
    && \mbox{{\rm subject to\; }} &&& \sum_{l=1}^{n}\sum_{m=1}^{n} \Gamma_{k, l, m} c(z_l,z_m)\leq r_k,\; k=1,2; \\
    &&&&& \Gamma_{1, l, m}\bm{1}_n= \mu_0,\ \Gamma_{2, l, m}\bm{1}_n= \nu_0;\\
            && &&& \Gamma_{1, l, m}^\intercal\bm{1}_n= p_1,\  \Gamma_{2, l, m}^\intercal\bm{1}_n= p_2.
  \end{aligned}
\end{equation}
\end{theorem}
\begin{proof}
See the Appendix.
\end{proof}
Note that the optimization problem in \eqref{eq:wass vs wass:dual} is the problem \eqref{eq:kl_min} for discrete distributions supported on the joint empirical samples $\{z_1,\ldots,z_n\}$. The variables $\Gamma_1,\Gamma_2$ are two matrices representing how the probability mass is transported between the empirical distributino $\mu_0,\nu_0$ and the desired LFD $p_1,p_2$. 
Instead of solving the infinite-dimensional problem \eqref{eq:kl_min}, we can now solve the finite-dimensional optimization problems \eqref{eq:wass vs wass:dual} which can be solved efficiently using off-the-shelf solvers. It is a linear program when $c$ is $\ell_1$ or $\ell_\infty$ norms, and a conic program when $c$ is $\ell_2$ norm, and the complexity quadratically depends on $n$.
It is worth mentioning that the equivalence between \eqref{eq:wass vs wass:dual} and \eqref{eq:kl_min} is not obvious and depends on the properties of the objective function and ambiguity sets. 

\subsection{Optimal Stopping Time}

Once we find $(\tilde \mu,\tilde \nu)$ by which the ambiguity sets are weakly stochastically bounded, the optimal stopping time that solves the problem \eqref{eq:Lorden} asymptocially can be constructed as the CUSUM procedure \cite{page-biometrica-1954} based on $(\tilde \mu,\tilde \nu)$. The detection statistic can be computed recursively as
\begin{equation}\label{eq:cusum}
S_t = (S_{t-1})^{+} + \log \frac{d\tilde \nu}{d\tilde \mu}(x_t), \quad S_0 = 0,
\end{equation}
and stopping time is therefore defined as
\begin{equation}\label{eq:cusum_stop}
\cT := \inf\{t: S_t \geq b\},
\end{equation}
where $b$ is a pre-specified threshold such that the average run length meets the desired lower bound $\gamma$.

\begin{theorem}[Asymptotical Optimality]\label{thm:opt}
Consider the ambiguity sets $\calP_{\mu_0},\calP_{\nu_0}$ in \eqref{eq:was_set_vanilla}, and suppose the pair of distributions $(\tilde \mu,\tilde \nu)$ found through \eqref{eq:wass vs wass:dual} satisfies the condition \eqref{eq:cond2}. Then the CUSUM test \eqref{eq:cusum}-\eqref{eq:cusum_stop} under $(\tilde \mu,\tilde \nu)$ with threshold $b=|\log \gamma|$ solves \eqref{eq:Lorden} and \eqref{eq:Pollak} asymptotically as $\gamma \rightarrow \infty$.
\end{theorem}
\begin{proof}
From Theorem \ref{thm:lfd}, the pair $(\tilde \mu,\tilde \nu)$ is a solution to \eqref{eq:kl_min}, which is equivalent to \eqref{eq:cond1}. If $(\tilde \mu,\tilde \nu)$ also satisfies \eqref{eq:cond2}, then the ambiguity sets $\calP_{\mu_0},\calP_{\nu_0}$ are weakly stochastically bounded by $(\tilde \mu,\tilde \nu)$. From Theorem 3 in \cite{molloy2017misspecified}, we have the desired results.
\end{proof}

%


Note that the CUSUM procedure as in \eqref{eq:cusum} with threshold $b=|\log \gamma|$ is asymptotically optimal for both Lorden's and Pollak's formulations when the true distribution is $\tilde \mu$ and $\tilde \nu$. The results in Theorem \ref{thm:opt} means that when we have two ambiguity sets, the CUSUM test based on the LFDs are minimax robust asymptotically optimal for the robust Lorden's \eqref{eq:Lorden} and Pollak's formulations \eqref{eq:Pollak}. When the true distributions differ from the LFDs, the price we pay in performance loss is due to the robustness that we would like to guarantee.



\subsection{Extensions and Modifications}\label{sec:extension}
We note that the LFDs found through \eqref{eq:wass vs wass:dual} is only supported on the historical data used to construct the empirical distributions. When applied to new observations that are outside the support of those empirical distributions, we need to modify the algorithm to make it applicable in real scenarios. Here we mention two possible methods. 

\subsubsection{Kernel convolution}\label{sec:kernel}
Firstly, we may interpolate the discrete LFDs within the entire sample space $\mathcal X$, through, for example, kernel convolution. And then apply the modified LFDs to calculate the detection statistics for new observations. A simple example is to convolve with the Gaussian kernel $K_h(x,y)=\exp\{-(x-y)^2/(2h^2)\}/\sqrt{2\pi h^2}$, with a carefully chosen kernal bandwidth. More specifically, the smoothed LFDs after convolution are
\[
\tilde \nu'(x) = \int \tilde \nu(y)K_h(x,y)dy,\ \tilde \mu'(x) = \int \tilde \mu(y)K_h(x,y)dy.
\]
Thus the detection statistic in \eqref{eq:cusum} becomes $S_t = (S_{t-1})^{+} + \log \{d\tilde \nu'(x_t)/d\tilde \mu'(x_t)\}$.

\subsubsection{Binning approach}\label{sec:bin} 
Second is a binning approach, which has been used previously in change-point detection problems but for different purposes \cite{lau2018binning}. In detail, we could partition the sample space $\mathcal X$ into $L$ exclusive and exhaustive regions, $\mathcal X_1,\ldots,\mathcal X_L$, satisfying $\mathcal X_i\cap \mathcal X_j=\emptyset$ and $\cup_{i=1}^L \mathcal X_i = \mathcal X$. In this way, we convert any continuous distribution into discrete ones and the LFDs can then be used naturally for new observations.

\section{Numerical results}\label{sec:numerical}

In this section, we investigate the performance of the proposed robust test based on Wasserstein ambiguity sets (which we call Robust-Was CUSUM in this section). 
For illustrative purposes, we consider a simple Gaussian mean shift example where the data distribution changes from $\mu=\mathcal N(0,1)$ to $\nu_m=\mathcal N(m,1)$ with the post-change mean $m$ takes two possible values $0.5$ and $1$, representing different signal-to-noise ratios. 
We randomly generate $50$ samples from the pre-change distribution $\mathcal N(0,1)$ and $50$ samples from the post-change distribution $\mathcal N(m,1)$. Then we construct the ambiguity sets based on the Wasserstein metric as shown in \eqref{eq:was_set_vanilla}. Then we solve the convex programming problem \eqref{eq:wass vs wass:dual} to find the LFDs $\tilde \mu$ and $\tilde \nu$. The Robust-Was CUSUM test is constructed based on the LFDs, according to the definition \eqref{eq:cusum} and \eqref{eq:cusum_stop}.

In the first result, we use the convolution approach to extend the LFDs to the whole sample space and then calculate the resulting CUSUM statistic. We compare the performance of the Robust-Was CUSUM test with the exact CUSUM and the GLR test. In detail, the exact CUSUM test is constructed assuming full knowledge of the true distributions, i.e., the exact CUSUM statistic is defined as in \eqref{eq:cusum} using true distributions $\nu$ and $\mu$. Moreover, the exact CUSUM test is the optimal test in the sense that it has the smallest detection delay and thus serves as the information-theoretic lower bound to the detection delay \cite{Lorden1971,moustakides1986optimal}. 
The GLR test is designed for the case when the post-change parameter $m$ is unknown. The parameter is estimated using maximum likelihood estimate and plugged into the log-likelihood ratio to calculate the GLR statistics. Moreover, to increase the efficiency, we adopt the window-limited GLR approach with the test statistic \cite{lai-ieeetit-1998}:
\[
S_t^{\text G}: = \max_{t-W\leq k \leq t} \max_{m\in \mathbb{R}}\sum_{i=k}^t \log\frac{d\nu_m}{d\mu}(x_i), 
\]
where $W$ is the window size and is chosen at $50$, the same as the number of empirical observations used in Robust-Was CUSUM.
The radii parameters $r_1$ and $r_2$ are set to be equal. We select smaller radii for smaller post-change mean, since the empirical samples tend to be closer as $m$ decreases and we need two ambiguity sets to have an empty intersection. 
The kernel bandwidth parameter as in Section \ref{sec:kernel} is chosen as $h=0.25$. Moreover, each time after solving the LFDs, we verify that the condition \eqref{eq:cond2} indeed holds.

We plot the expected detection delay versus average run length for different methods, averaged over 10000 times, as shown in Fig. \ref{fig:robust_cusum}. 
We see that the robust CUSUM derived from the Wasserstein ambiguity sets has a smaller delay than the GLR test.

\begin{figure}[!ht]
\centering
\begin{tabular}{cc}
\includegraphics[width=0.48\linewidth]{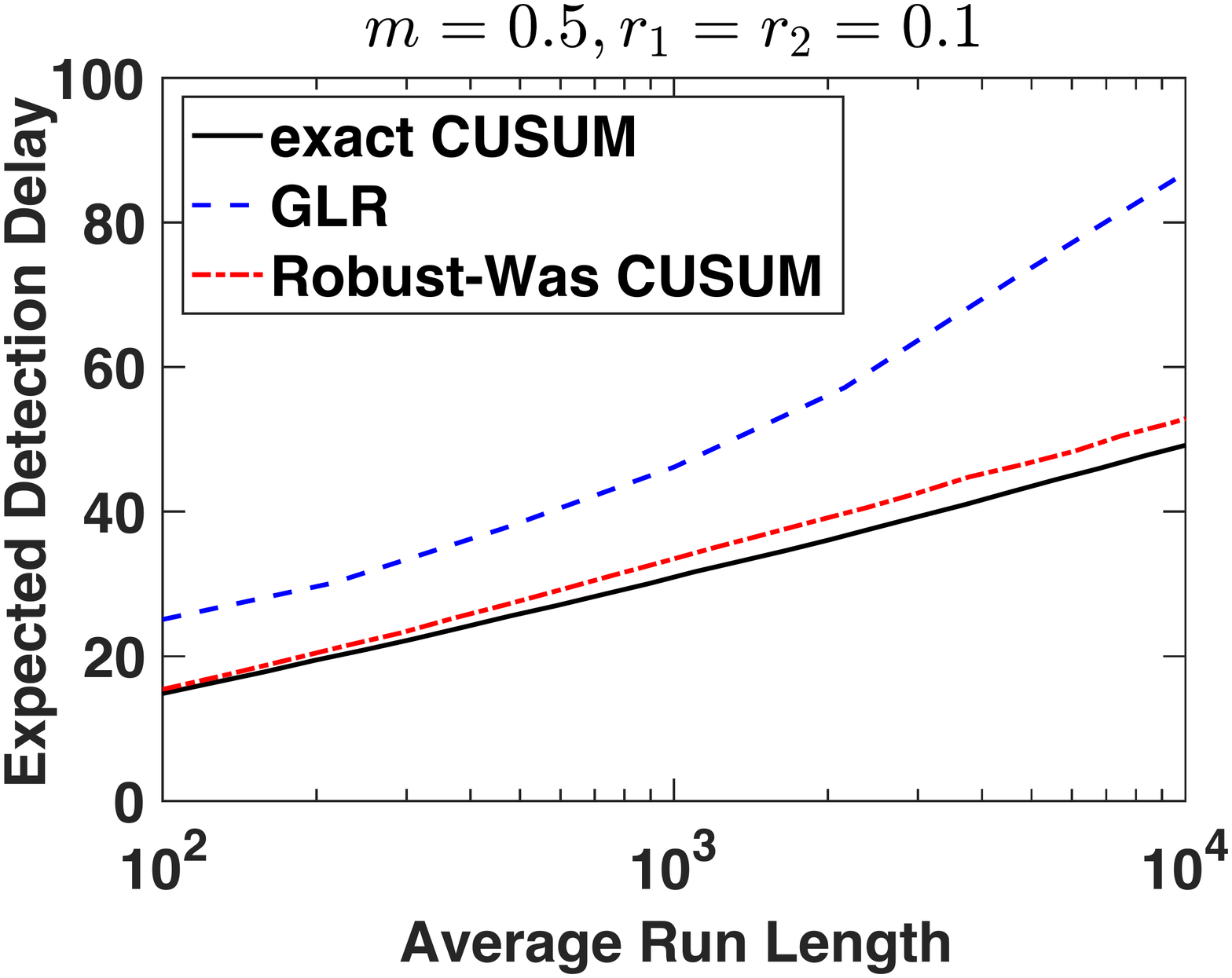} &
\includegraphics[width=0.48\linewidth]{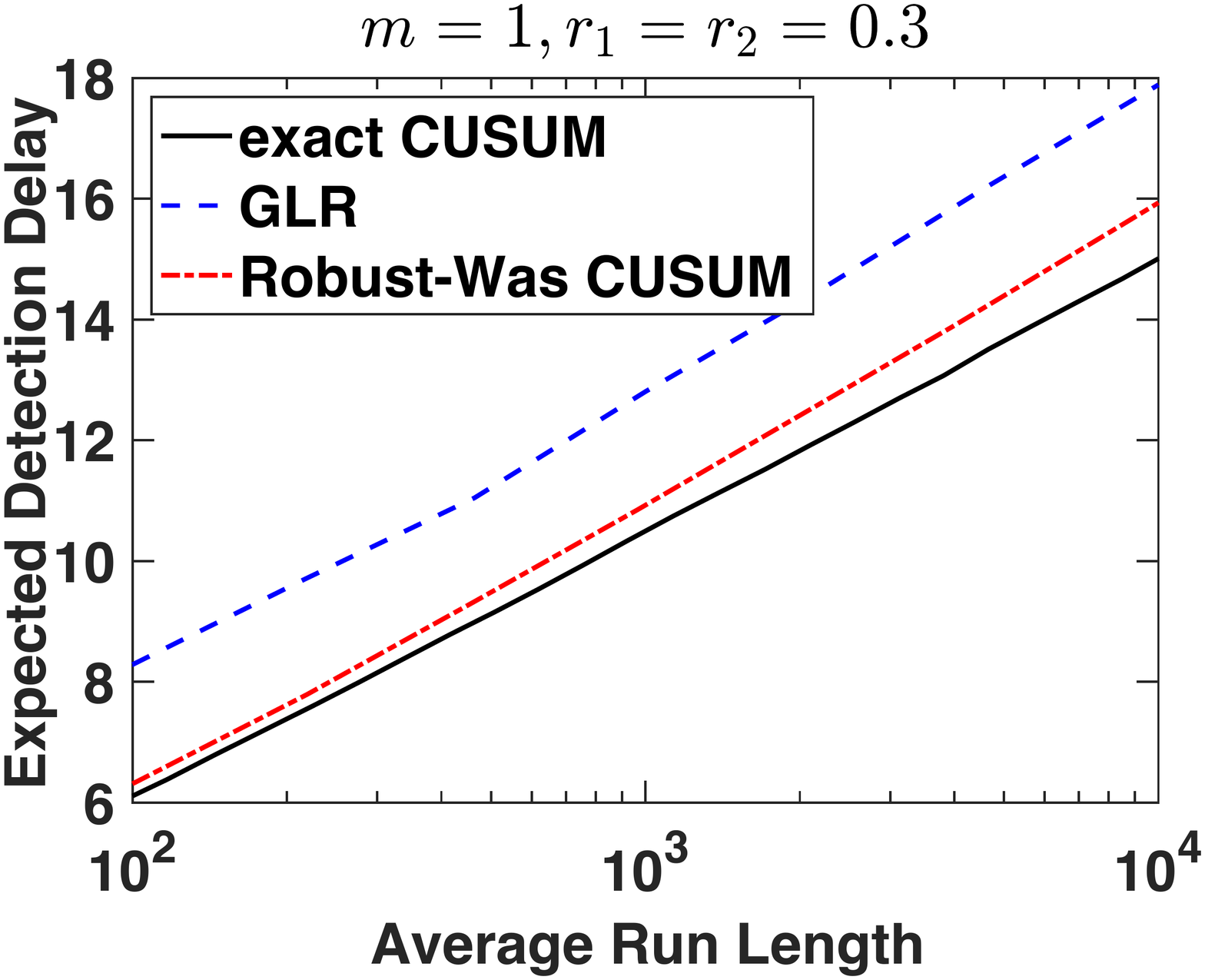} \\
{\small $m=0.5,r_1=r_2=0.1$} & {\small $m=1,r_1=r_2=0.3$}
\end{tabular}
\caption{The detection delay comparison with exact CUSUM and GLR, under different signal-to-noise ratios.}
\label{fig:robust_cusum}
\vspace{-0.1in}
\end{figure}
 
We then compare the performance of the proposed method under the binning approach detailed in Section \ref{sec:bin}, with bin size $L=20$. We select the breakpoints such that the resulting discretized pre-change distributions is a uniform distribution. In such case, we can also compare with the robust CUSUM test based on KL ambiguity sets \cite{molloy2017misspecified}, where the two ambiguity sets are constructed using the KL divergence and LFDs are again found through \eqref{eq:cond1} and \eqref{eq:cond2}. The detection delay shown in Fig. \ref{fig:robust_cusum_bin} shows that the KL robust CUSUM test tends to have a larger detection delay and the proposed Robust-Was CUSUM test still has a better performance.

\begin{figure}[!ht]
\centering
\begin{tabular}{cc}
\includegraphics[width=0.48\linewidth]{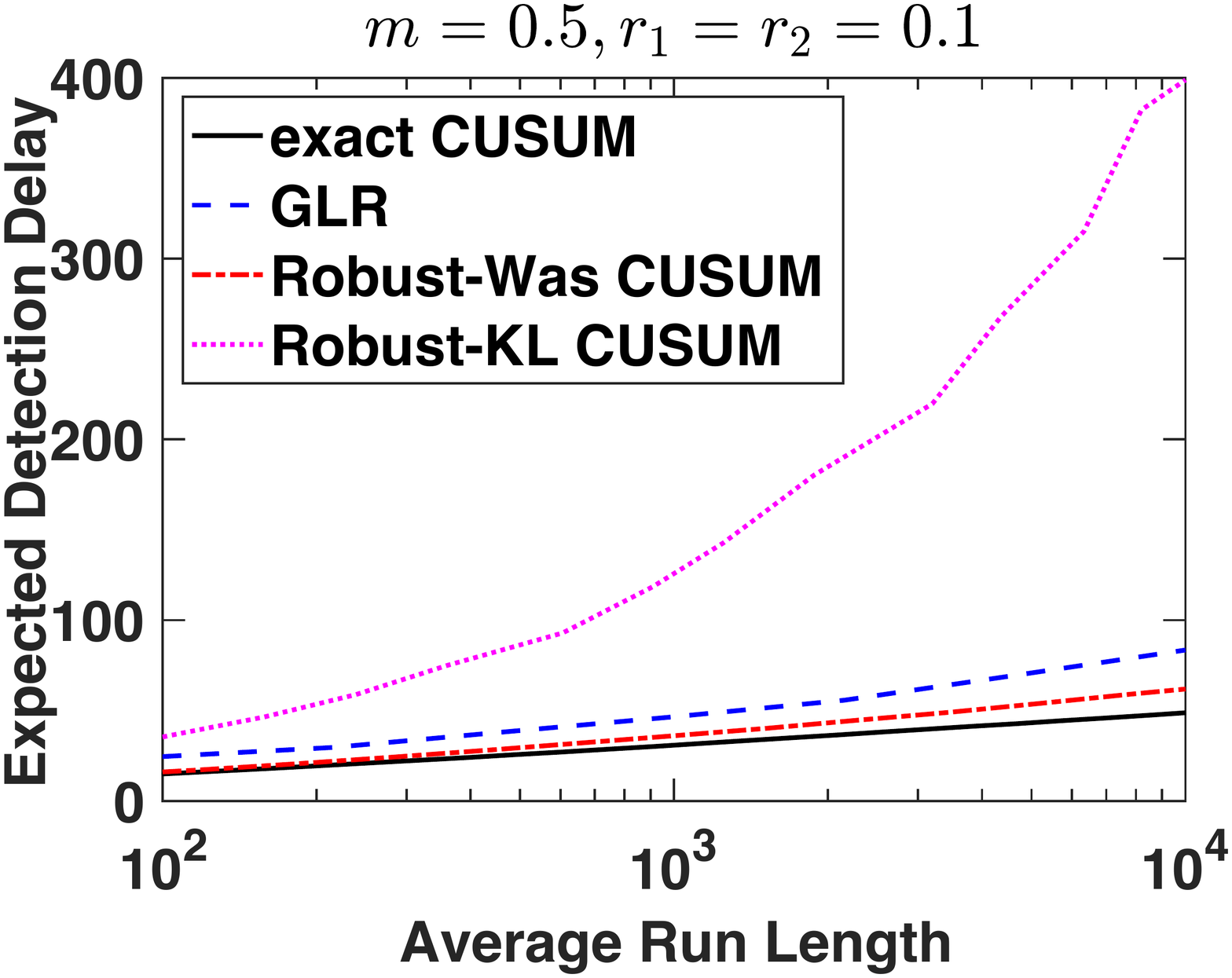} &
\includegraphics[width=0.48\linewidth]{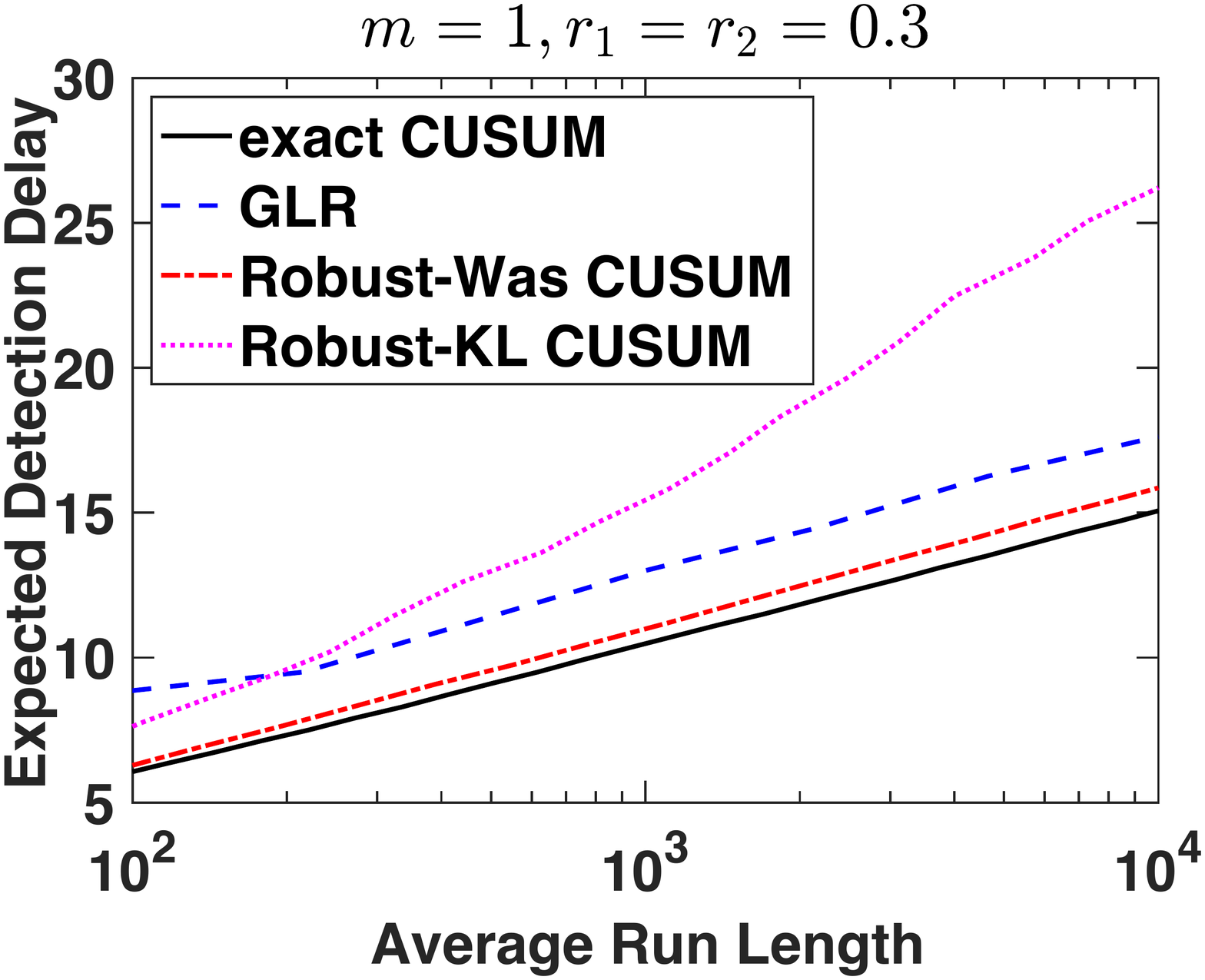} \\
{\small  $m=0.5$ } & {\small  $m=1$}
\end{tabular}
\caption{The detection delay comparison with exact CUSUM, GLR, and the robust test under KL ambiguity sets, after data binning.}
\label{fig:robust_cusum_bin}
\vspace{-0.1in}
\end{figure}

We also compare the performance when the observations are contaminated. We add a uniform noise (contamination) into the observations. The contamination follows the uniform distributions on the interval $[-\epsilon,0]$. We test five values for $\epsilon$ from $0.1$ to $0.5$, representing different strength levels of the contamination. The average detection delay is plotted in Fig. \ref{fig:robust_cusum_bin_conta}. The exact CUSUM algorithm is no longer optimal when the observations are contaminated, since there is a mismatch between the distribution used to construct CUSUM statistics and the true data distribution after contamination. From Fig. \ref{fig:robust_cusum_bin_conta}, we see that the proposed method may even have a smaller detection delay than the exact CUSUM method.

\begin{figure}[!ht]
\centering
\includegraphics[width=0.6\linewidth]{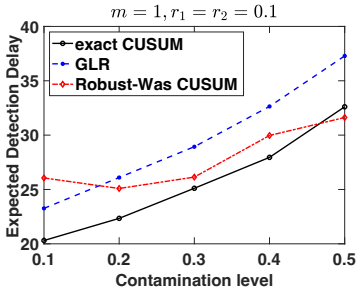}
\vspace{-0.1in}
\caption{The detection delay comparison with exact CUSUM and GLR, after data binning and for contaminated data. ARL fixed at 50000.}
\label{fig:robust_cusum_bin_conta}
\vspace{-0.1in}
\end{figure}


\section{Conclusions and Discussions}\label{sec:conclusion}
We applied Wasserstein ambiguity sets to robust quickest change detection. This also brought new questions worth investigating. First, it would be of great importance to study a data-driven and precise characterization of the radii for future work. Second, the real data are usually under contamination; thus it would be interesting to study the theoretical performance of the robust test under contaminated data or outliers. Third, the LFDs solved in this work are discrete distributions, it would be 
worthwhile to study the theoretical loss or explore different ways to better fix such problem. 
\section{Appendix}

\begin{proof}[Proof of Theorem \ref{thm:lfd}]
Denote by $L^1(\mu)$ the space of all integrable functions with respect to the measure $\mu$. Using the Kantorovich duality \cite{villani2008optimal}, the Wasserstein distance equals:
\[
\wass(\mu,\nu)=\sup_{\substack{ (\phi,\psi)\in L^1(\mu) \times L^1(\nu) \\ \phi(x) + \psi(y) \leq c(x,y) \\ \forall x,y}} \left( \int_{\mathcal X} \phi(x)d\mu + \int_{\mathcal X} \psi(y)d\nu \right).
\]
Following \cite{gao2018robust}, we rewrite the problem using the Lagrangian of the optimization problem \eqref{eq:kl_min} and the above equation:
\[
\begin{aligned}
&\inf_{\mu,\nu\in\scrP(\Omega)}\sup_{\substack{\lambda_1,\lambda_2 \geq 0\\u_1\in \mathbb R^{n_1}, u_2\in \mathbb R^{n_2}\\v_1\in L^1(\mu),v_2\in L^1(\nu)}} \Big\{\mathsf{KL}(\nu||\mu) - \lambda_1 r_1 - \lambda_2 r_2 +  \\
& \frac{1}{n_1}\sum_{i=1}^{n_1} u_1^i + \frac{1}{n_2}\sum_{i=1}^{n_2} u_2^i +  \int_\mathcal X v_1(x) d\mu + \int_\mathcal X v_2(x) d\nu : \\
&u_1^i+v_1(\xi)\leq \lambda_1 c(\xi,x_i),\ \forall 1\leq i\leq n_1, \forall \xi \in \mathcal X, \\
& u_2^i+v_2(\xi)\leq \lambda_2 c(\xi,y_i),\ \forall 1\leq i\leq n_2, \forall \xi \in \mathcal X \Big\}.
\end{aligned}
\]
Furthermore, since the objective function is increasing in $v_1,v_2$, we can replace $v_1$ with $\min_{1\leq i\leq n_1} \{\lambda_1 c(\xi,x_i) - u_1^{i}\}$ and replace $v_2$ with $\min_{1\leq i\leq n_2} \{\lambda_1 c(\xi,y_i) - u_2^{i}\}$. Interchanging $\sup$ and $\inf$, we have
\[
\begin{aligned}
&\inf_{\mu\in \calP_{\mu_0}, \nu\in\calP_{\nu_0}} \mathsf{KL}(\nu||\mu)\\
\geq &\!\!\!\sup_{\substack{\lambda_1,\lambda_2 \geq 0\\u_1\in \mathbb R^{n_1}, u_2\in \mathbb R^{n_2}}} \bigg\{- \lambda_1 r_1 - \lambda_2 r_2 +  \frac{1}{n_1}\sum_{i=1}^{n_1} u_1^i + \frac{1}{n_2}\sum_{i=1}^{n_2} u_2^i +\\ &\inf_{\mu,\nu\in\scrP(\mathcal X)}  \Big\{ \mathsf{KL}(\nu||\mu) + \int_\mathcal X \min_{1\leq i\leq n_1} \{\lambda_1 c(\xi,x_{i}) - u_1^{i}\} d\mu(\xi) \\
&\hspace{50pt} + \int_\mathcal X \min_{1\leq i\leq n_2} \{\lambda_2 c(\xi,y_{i}) - u_2^{i}\} d\nu(\xi) \Big\} \bigg\}.    
\end{aligned}
\]
For the inner infimum problem, note that $\forall(\mu,\nu)$ and $\forall \xi\in \supp(\mu)\cup\supp(\nu)$, let $i_1(\xi) = \argmin_{i}\{\lambda_1 c(\xi,x_{i}) - u_1^{i}\}$, $i_2(\xi) = \argmin_{i}\{\lambda_2 c(\xi,y_{i}) - u_2^{i}\}$ , set
\[
\begin{aligned}
T(\xi) 
:=& \begin{cases} x_{i_1(\xi)}, & \text{if } \lambda_1d\mu(\xi) \geq \lambda_2d\nu(\xi),  \\
y_{i_2(\xi)}, & \text{if } \lambda_1d\mu(\xi) < \lambda_2d\nu(\xi),
\end{cases}
\end{aligned}
\]
then $T(\xi)$ belongs to the minimum of $\min_{1\leq i\leq n_1} \{\lambda_1 c(\xi,x_{i}) - u_1^{i}\} d\mu(\xi) +  \min_{1\leq i\leq n_2} \{\lambda_2 c(\xi,y_{i}) - u_2^{i}\} d\nu(\xi)$.
Moreover, construct another distributions $(\mu',\nu')$ such that $\mu'(B)=\mu\{\xi\in\mathcal X:T(\xi)\in B\}$ and $\nu'(B)=\nu\{\xi\in\mathcal X:T(\xi)\in B\}$ for any Borel set $B\subset \hat{\mathcal X}:=\{z_1,\ldots,z_n\}$. Then it is easy to see that
\[
\begin{aligned}
&\int_{\hat{\mathcal X}} \min_{1\leq i\leq n_1} \{\lambda_1 c(\xi,x_{i}) - u_1^{i}\} d\mu'(\xi) +\\
& \hspace{50pt} \int_{\hat{\mathcal X}}  \min_{1\leq i\leq n_2} \{\lambda_2 c(\xi,y_{i}) - u_2^{i}\} d\nu'(\xi)\\
\leq & \int_\mathcal X \min_{1\leq i\leq n_1} \{\lambda_1 c(\xi,x_{i}) - u_1^{i}\} d\mu(\xi)  + \\
& \hspace{50pt} \int_\mathcal X \min_{1\leq i\leq n_2} \{\lambda_2 c(\xi,y_{i}) - u_2^{i}\} d\nu(\xi)
\end{aligned}
\]
In addition, for $\nu$ that is absolutely continuous with respect to $\mu$, we have 
$\mathsf{KL}(\nu'||\mu')\leq \mathsf{KL}(\nu||\mu)$ since the KL divergence is a convex function.


Hence $(\mu',\nu')$ yields an objective value no worse than $(\mu,\nu)$ for the inner infimum problem. This means that it suffices to only consider $(\mu,\nu)$ supported on the empirical set $\hat{\mathcal X}$. Following a similar argument as in \cite{gao2018robust},
the optimization problem can be reduced to a finite-dimensional convex optimization problem as shown in Theorem \ref{thm:lfd}. 
\end{proof}

\clearpage
\bibliographystyle{IEEEtran}
\bibliography{ref}

\end{document}